\documentclass{file}
\usepackage{pat}
\usepackage{paralist}
\usepackage{dsfont}  
\usepackage[utf8x]{inputenc}
\usepackage{skull}
\usepackage{paralist}
\usepackage{centernot}
\usepackage{mathtools}
\newcounter{algsubstate}
\renewcommand{\thealgsubstate}{\alph{algsubstate}}
\newenvironment{algsubstates}
  {\setcounter{algsubstate}{0}%
   \renewcommand{\State}{%
     \stepcounter{algsubstate}%
     \Statex {\footnotesize\thealgsubstate:}\space}}
  {}
\usepackage[shortlabels]{enumitem}
\usepackage{stmaryrd}
\usepackage{graphicx}
\usepackage[most]{tcolorbox}
\usepackage{algpseudocode}
\usepackage{varwidth}
\usepackage{algorithm}
  \makeatletter
  \renewcommand{\ALG@name}{Procedure}
 \makeatother

\makeatother
\usepackage{bbm}  
\usepackage{listings}
\usepackage{pifont}%
\usepackage[normalem]{ulem}
\usepackage{cancel}

\usepackage{todonotes}

\usepackage{etoolbox}
\usepackage{float}
\usepackage{tabularx}
\usepackage{array}
\usepackage{subfig}
\usepackage{graphicx}
\usepackage{multirow}
\usepackage{makecell}
 \usepackage{relsize}
\newtheorem{observation}{Observation}
\newtheorem{proposition}{Proposition}
\newtheorem{definition}{Definition}
\newtheorem{theorem}{Theorem}
\newtheorem{lemma}{Lemma}
\newtheorem{openproblem}{Problem}

\newtheorem{notation}{Notation}
\newtheorem{corollary}{Corollary}
\newtheorem{claim}{Claim}

\definecolor{maroon}{rgb}{0.5, 0.0, 0.0}
\definecolor{darkblue}{rgb}{0.0, 0.0, 0.55}

\newcommand{\crefdefpart}[2]{\hyperref[#2]{\namecref{#1}~\labelcref*{#1}~\ref*{#2}}}
\algdef{SE}[SUBALG]{Indent}{EndIndent}{}{\algorithmicend\ }%
\algtext*{Indent}
\algtext*{EndIndent}

\usepackage[longnamesfirst,numbers,sort&compress]{natbib}

\usepackage[mathlines]{lineno}
\newcommand*\patchAmsMathEnvironmentForLineno[1]{%
 \expandafter\let\csname old#1\expandafter\endcsname\csname #1\endcsname
 \expandafter\let\csname oldend#1\expandafter\endcsname\csname end#1\endcsname
 \renewenvironment{#1}%
    {\linenomath\csname old#1\endcsname}%
    {\csname oldend#1\endcsname\endlinenomath}}%
\newcommand*\patchBothAmsMathEnvironmentsForLineno[1]{%
 \patchAmsMathEnvironmentForLineno{#1}%
 \patchAmsMathEnvironmentForLineno{#1*}}%
\AtBeginDocument{%
\patchBothAmsMathEnvironmentsForLineno{equation}%
\patchBothAmsMathEnvironmentsForLineno{align}%
\patchBothAmsMathEnvironmentsForLineno{flalign}%
\patchBothAmsMathEnvironmentsForLineno{alignat}%
\patchBothAmsMathEnvironmentsForLineno{gather}%
\patchBothAmsMathEnvironmentsForLineno{multline}%
}

\definecolor{brightmaroon}{rgb}{0.76, 0.13, 0.28}
\definecolor{linkblue}{rgb}{0, 0.337, 0.227}

\newrobustcmd{\onesub}{\mathord{\includegraphics{figs/one-sub}}}
\newrobustcmd{\leftup}{\mathord{\includegraphics{figs/left-up}}}
\usepackage{thm-restate}
\makeatletter
\newcommand{\xMapsto}[2][]{\ext@arrow 0599{\Mapstofill@}{#1}{#2}}
\def\Mapstofill@{\arrowfill@{\Mapstochar\Relbar}\Relbar\Rightarrow}
\makeatother
\newenvironment{txteq}
{
	\begin{equation}
	\begin{minipage}[t]{0.85\textwidth} 
	\em                                
}
{\end{minipage}\end{equation}\ignorespacesafterend}
\newenvironment{txteq*}
{
	\begin{equation*}
	\begin{minipage}[t]{0.85\textwidth} 
	\em                                
}
{\end{minipage}\end{equation*}\ignorespacesafterend}

\title{Sparse graphs with local covering conditions on edges}
\author{
Debsoumya Chakraborti\thanks{
Mathematics Institute, University of Warwick, Coventry, United Kingdom. 
Supported by the European Research Council (ERC) under the European Union Horizon 2020 research and innovation programme (grant agreement No. 947978). 
Email: {\tt debsoumya.chakraborti@warwick.ac.uk}. 
}
\;\;\; Amirali Madani\thanks{School of Computer Science, Carleton University, Ottawa, Ontario, Canada. Supported by the Natural Sciences and Engineering Research Council of Canada (NSERC). Emails: {\tt amiralimadani@cmail.carleton.ca} and {\tt anil@scs.carleton.ca}}
\;\;\; Anil Maheshwari\footnotemark[2] \;\;\; Babak Miraftab\thanks{School of Computer Science, Carleton University, Ottawa, Ontario, Canada. Email: {\tt bobby.miraftab@gmail.com}.
}
}

\def \G {\mathcal{G}_{tree}}

\date{}

\begin{document}

\maketitle
\begin{abstract}
In 1988, Erd\H{o}s suggested the question of minimizing the number of edges in a connected $n$-vertex graph where every edge is contained in a triangle. 
Shortly after, Catlin, Grossman, Hobbs, and Lai resolved this in a stronger form. 
In this paper, we study a natural generalization of the question of Erd\H{o}s in which we replace `triangle' with `clique of order $k$' for ${k\ge 3}$. 
We completely resolve this generalized question with the characterization of all extremal graphs. 
Motivated by applications in data science, we also study another generalization of the question of Erd\H{o}s where every edge is required to be in at least $\ell$ triangles for $\ell\ge 2$ instead of only one triangle. 
We completely resolve this problem for $\ell = 2$. 
\end{abstract}

\section{Introduction}
In 1988, Erd\H{o}s~\cite{erdos} posed the following question: What is the minimum number of edges in a connected $n$-vertex graph such that every edge is part of a triangle? 
Note that without the connectedness assumption, the question is trivial to answer because the empty graph satisfies the hypothesis. 
A few years later, Catlin, Grossman, Hobbs, and Lai~\cite{CATLIN1992285} considered a more general question and showed that every $n$-vertex graph with $c$ connected components where every edge belongs to a triangle has at least $\frac{3}{2}(n-c)$ edges.

A natural extension of the question of Erd\H{o}s is to determine the minimum number of edges in a connected $n$-vertex graph where each edge belongs to a copy of $K_k$ for a given $k\ge 3$. In this paper, we completely resolve this question. We also study the problem of minimizing the number of edges in a connected $n$-vertex graph where every edge is contained in at least $\ell$ triangles. For brevity, we say a graph $G$ has a \defin{$(k,\ell)$-cover} if every edge of $G$ lies in at least $\ell$ copies of $K_k$. Before diving into our main results, we briefly mention a few related works.

\subsection{Relevant literature}
\label{secliter}
An extremal problem similar to this paper was considered in \cite{chakraborti2020extremal}, where the authors studied an edge minimization problem of graphs with a different condition. Instead of requiring every edge to be a part of a copy of $K_k$, they imposed the condition that every vertex is contained in a copy of a $K_k$. 
\begin{proposition}[Proposition~1.1 in \cite{chakraborti2020extremal}] \label{prop:previous result}
Let $k\ge 2$ and $G$ be an $n$-vertex graph such that every vertex is in a copy of $K_k$ and $n-k = qk + r$ where $q\ge 0$ and $1\le r\le k$. 
Then, the number of edges in $G$ is at least $(q+2)\binom{k}{2} - \binom{k-r}{2}$. Moreover, equality is achieved if and only if $G$ is a graph consisting of the union of 2 copies of $K_k$ sharing $k-r$ vertices, together with the disjoint union of $q$ copies of $K_k$. 
\end{proposition}

Another motivation for this paper comes from recent studies of \emph{cohesive} subgraphs in the data science literature. Loosely speaking, for a given graph, a cohesive subgraph is one in which the vertices are \emph{densely} connected to each other~\cite{corekeywordcommunitysearch,fomin2023building,chitnis}. Since graph-based representations are common in massive data analysis, cohesive subgraphs are useful in many different areas of big data. Over the years, finding and maintaining such subgraphs have found applications in various fields such as community search~\cite{corekeywordcommunitysearch,corekeywordcommunitysearch2,trusskeywordcommunitysearch, trusskeywordcommunitysearch2, trusskeywordcommunitysearch3} (e.g., finding researchers with similar interests in collaboration networks), E-commerce~\cite{ecommerce} (e.g., finding and matching similar items to similar users), and Biology~\cite{biology}. Cohesive subgraphs can be defined using different cohesion measures. Among these variants, $\ell$-truss~\cite{cohen,trusskeywordcommunitysearch,trusskeywordcommunitysearch2, trusskeywordcommunitysearch3,trussextra,trussextra2} is one of the most commonly used ones. 
An \defin{$\ell$-truss} of a graph $G$ is a maximal connected subgraph of $G$ in which every edge is contained in at least $\ell$ triangles. 
This motivates the study of graphs with  $(3,\ell)$-covers. Indeed, this was already studied by Burkhardt, Faber, and Harris~\cite{BurkhardtFH22,BurkhardtFH20} who showed that the minimum number of edges in a connected $n$-vertex graph with a $(3,\ell)$-cover is $n\left(1+\frac{\ell} {2}\right)+O\left(\ell^2\right)$.

\subsection{Main results}
\label{mainresults}
We start with some basic notations and definitions.
\begin{notation}
For a natural number~$m$, we write $[m]$ to denote the set $\{1,\dots,m\}$.

Consider a (hyper)graph $G$. We will denote by $V(G)$ and $E(G)$ the vertex set and the edge set of $G$, respectively. 
We will sometimes write $G=(V,E)$, where $V=V(G)$ and $E=E(G)$. 
For $X\subseteq V(G)$, we denote by $G[X]$ the subgraph of $G$ induced by $X$. 
For standard graph-theoretic notations and definitions, we refer the readers to~\cite{bondy1982graph}. 

A hypergraph is called \defin{linear} if every pair of distinct hyperedges intersect in at most one vertex. A connected hypergraph is called a \defin{hypertree} if it contains no cycles. Here, for a given hypergraph, a \defin{cycle} of length $\ell\geq 2$ is a subhypergraph with $\ell$ hyperedges that can be labelled as $e_1,\dots,e_{\ell}$ such that there exist distinct vertices $v_1,\ldots,v_{\ell}$ with $v_i\in e_i\cap e_{i+1}$ for every $i\in [\ell]$ (identifying $e_{\ell +1}$ with $e_1$).
\end{notation}
For the convenience of classifying extremal graphs of our results, we present the following definition. 
This definition can be seen as a special case of the standard notion of tree-decomposition (see, e.g., Section~12.3 in \cite{diestel}). 
\begin{definition}
    For a family of graphs $F_1,\dots,F_m$, define $\G(F_1,\dots,F_m)$ to be the family of graphs $G$ such that 
    there is a linear hypertree with the vertex set $V(G)$ and $m$ hyperedges $E_1,\dots,E_m\subseteq V(G)$, one corresponding to every graph $F_i$ satisfying the following. 
        \begin{enumerate}
        \item For every edge $uv\in E(G)$, there exists $i\in [m]$ such that $u,v\in E_i$, and  
		\item for every~$i\in [m]$, the graph $G[E_i]$ induced by $E_i$ is isomorphic to $F_i$.
	\end{enumerate}
    \label{def_tree}
\end{definition}
Note that the number of vertices in any graph in $\G(F_1,\dots,F_m)$ is ${1 + \sum_{i=1}^m (|E_i|-1)}$, and the number of edges in such graphs is $\sum_{i=1}^m |E(F_i)|$. 
Next, we give an example of a graph in $\G(F_1,\dots,F_m)$, to help digest the above definition. By identifying a special vertex in each $F_i$ and then gluing every $F_i$ together on that special vertex, we obtain a graph in $\G(F_1,\dots,F_m)$, where the underlying linear hypertree is a star. \Cref{exampletightcliquefigure} demonstrates an example of such a graph in the family $\G(K_4,K_4,L)$, where $L$ is the graph obtained from the union of 2 copies of $K_4$ sharing $2$ vertices. In such situations, we use the shorthand notation $\G(2K_4,L)$ to denote the same family. 
\begin{figure}[H]
    \centering
        {\includegraphics[scale=0.6]{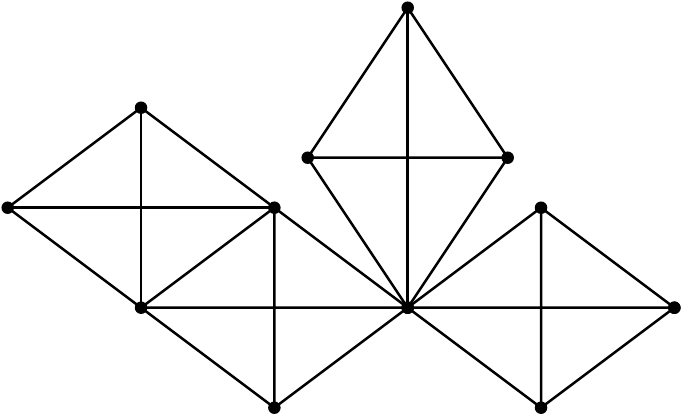}}
    \caption{A graph in the family $\G(2K_4,L)$. This connected graph has $12$ vertices, a \mbox{$(4,1)$-cover}, and $23$ edges which is the minimum possible.} 
    \label{exampletightcliquefigure}
\end{figure}

We are now ready to state our main result on the minimum number of edges in a connected $n$-vertex graph with a $(k,1)$-cover. 
Note that there is no connected $n$-vertex graph with a $(k,1)$-cover when $2\le n \le k-1$ and there is a unique connected $k$-vertex graph $K_k$ that has a $(k,1)$-cover. Thus, we assume the number of vertices to be more than $k$ in our first main result below.

\begin{theorem} \label{thm:1}
Let $k\ge 3$ and $G$ be a connected $n$-vertex graph with a \mbox{$(k,1)$-cover} such that $n-k = q (k-1) +r$ where $q\ge 0$ and $1\le r\le k-1$. 
Then, the number of edges in $G$ is at least $(q+2)\binom{k}{2} - \binom{k-r}{2}$. Moreover, equality is achieved if and only if $G\in \G(qK_k,L)$, where the graph $L$ is the union of 2 copies of $K_k$ sharing $k-r$ vertices. 
\end{theorem}
To see that the graphs in $\G(qK_k,L)$ achieve equality, note that the number of edges in $L$ is $2\binom{k}{2} - \binom{k-r}{2}$. For an illustration,  see \Cref{exampletightcliquefigure}.  
We note that the extremal graphs in \Cref{thm:1} are made by taking the connected components of the extremal graph in \Cref{prop:previous result} and then gluing them to form a tree-like structure. 

Analogous to the generalization of Catlin, Grossman, Hobbs, and Lai~\cite{CATLIN1992285} of Erd\H{o}s' question as mentioned before, we generalize \Cref{thm:1} to the situation where $G$ has multiple connected components. Before stating the result, we mention a small example. For $c\ge 1$ and $k\ge 3$, there is a unique $(k+c-1)$-vertex graph with $c$ connected components such that each component has a $(k,1)$-cover, namely the graph formed by taking the disjoint union of $c-1$ copies of $K_1$ and a copy of $K_k$ (which is simply a clique of order $k$ along with $c-1$ isolated vertices). 
\begin{corollary}
Let $c\ge 1$ and $k\ge 3$ and $G$ be an $n$-vertex graph with connected components $G_1,\dots, G_c$ such that every component has a \mbox{$(k,1)$-cover} with $n-k-c+1 = q (k-1) +r$, where $q\ge 0$ and $1\le r\le k-1$. 
Then, the number of edges in $G$ is at least $(q+2)\binom{k}{2} - \binom{k-r}{2}$. Moreover, the equality is achieved if and only if $c-1$ many components $G_i\in \{K_1\}\cup (\cup_{j\in \mathbb{N}} \G(jK_k))$ and the remaining component $G_i\in \cup_{j\in \mathbb{N}} \G(jK_k,L)$, where the graph $L$ is the union of 2 copies of $K_k$ sharing $k-r$ vertices.
\end{corollary}
We omit the proof of this since it is straightforward to apply \Cref{thm:1} to each connected component and apply convexity inequalities (e.g., \Cref{corollaryconvex}) to deduce it. 

The bound in \Cref{thm:1} also gives an upper bound on the minimum number of edges in a connected $n$-vertex graph with a \mbox{$(3,k-2)$-cover} because of the following observation. This improves the upper bound of the corresponding result of Burkhardt, Faber, and Harris in~\cite{BurkhardtFH20}.
\begin{observation}\label{obs}
    Every graph with a \mbox{$(k,1)$-cover} also has a \mbox{$(3,k-2)$-cover} for $k\ge 3$.
\end{observation}

It is natural to ask if the aforementioned upper bound obtained from using \Cref{thm:1,obs} is tight for graphs with $(3,\ell)$-covers for $\ell \geq 2$.
\begin{openproblem}\label{op4}
For $\ell \geq 2$, is the minimum number of edges for a connected $n$-vertex graph with a $(3,\ell)$-cover the same as the minimum number of edges in a connected $n$-vertex graph with a $(\ell+2,1)$-cover?
\end{openproblem}
Unfortunately, this question turns out to be false in general.
To see that, let $\ell= 2\ell'\ge 6$. 
By \Cref{thm:1}, the minimum number of edges for any connected $(2\ell'+4)$-vertex graph with a $(2\ell'+2,1)$-cover is $2\binom{2\ell'+2}{2} - \binom{2\ell'}{2}$. However, there exists a connected $(2\ell' + 4)$-vertex graph with a $(3,2\ell')$-cover and $4\binom{\ell'+2}{2} < 2\binom{2\ell'+2}{2} - \binom{2\ell'}{2}$ edges. Indeed, this is achieved by the complete $(\ell'+2)$-partite graph with $2$ vertices in each partition, i.e., the graph obtained from the clique $K_{2\ell'+4}$ after removing a perfect matching (see \cref{rmrk3fig1}).

\begin{figure}[H]
    \centering
        {\includegraphics[scale=0.70]{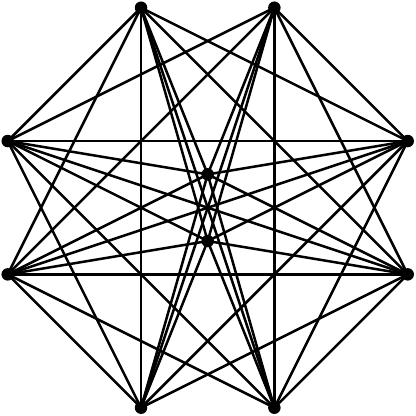}}
    \caption{An example of a connected $10$-vertex graph with a $(3,6)$-cover with $40$ edges}
    \label{rmrk3fig1}
\end{figure}
However, we show that for $\ell=2$, the answer to \Cref{op4} turns out to be affirmative. This follows from the following result along with \Cref{obs}. 
\begin{theorem} \label{thm:2}
If $G$ is a connected $n$-vertex graph with a \mbox{$(3,2)$-cover} with the minimum possible number of edges, then $G$ has a \mbox{$(4,1)$-cover}. 
\end{theorem}
Determining the exact minimum number of edges for a connected $n$-vertex graph with a $(3,\ell)$-cover remains open for $\ell\ge 3$. See \Cref{sec:concluding remarks} for a brief discussion on a couple of related open problems.

\section{Proof of \Cref{thm:1}}\label{sectionkone}

In this section, we prove a tight lower bound on the number of edges in any connected $n$-vertex graph with a $(k,1)$-cover for $k\geq 3$. 
We will use the following standard convexity inequality in our proof.

\begin{lemma}
    Suppose $m$, $r'$, $q'$, and $I$ are nonnegative integers with $r' < m$ and $q' < I$. Over all choices of values $x_j \in\{0,1,2, \dots, m\}$ for each $j \in [I]$ such that $\sum_{j=1}^ I x_j=q' m+r'$, the expression $\sum_{j=1}^{I}\binom{x_j +1}{2}$ achieves its maximum if and only if except at most one value of $j$, we have $x_j\in \{0, m\}$ (thus, $x_j=m$ for $q'$ many values of $j$). 
    \label{corollaryconvex}
\end{lemma}
\begin{proof} 
Let $f: \mathbb{Z}\rightarrow \mathbb{R}$ be a function with $f(n)= \frac{n(n+1)}{2}$ for every $n\in \mathbb{Z}$. Since $f$ is a strictly convex function, we have $f(n_1)+f(n_2)<f(n_1-1)+f(n_2+1)$ for all integers $n_1\le n_2$. 
Now assume to the contrary that there exists an assignment of values to the variables $x_j$ other than the one described in the statement of the lemma that maximizes $\sum_{j=1}^{I}\binom{x_j+1}{2}$. Then, this assignment must set two distinct variables $x_{j_1}=n_1$ and $x_{j_2}=n_2$ with $0<n_1\leq n_2<m$. By setting $x_{j_1}=n_1-1$ and $x_{j_2}=n_2+1$ and keeping the other variable assignments intact, $\sum_{j=1}^{I}x_j$ remains unchanged but $\sum_{j=1}^{I}\binom{x_j+1}{2}$ is strictly increased, a contradiction.
\end{proof}

We will need the following alternative description of the extremal graphs in \Cref{thm:1}, which is easy to verify. 
\begin{observation}\label{obs:extremal}
Let $k$, $n$, $q$, and $r$ be as in \Cref{thm:1}. 
Let $G$ be an $n$-vertex connected graph with $q+2$ subgraphs $C_0,C_1,\dots,C_{q+1}$, each of which is a copy of~$K_k$ such that letting $U_i := V(C_i)$ for every $i\in \{0,\dots,q+1\}$, we have 
\begin{enumerate}
    \item $E(G) = \cup_{i=0}^{q+1} E(C_i)$, and
    \item\label{enm:intersection one} except at most one value of $j\in [q+1]$, we have $|U_j \cap (\cup_{i=0}^{j-1} U_i)| = 1$, and
    \item if $j\in [q+1]$ violates \Cref{enm:intersection one}, then $|U_j \cap (\cup_{i=0}^{j-1} U_i)| = k-r$ and there exists $i\in \{0,\dots,j-1\}$ such that $|U_i\cap U_j| = k-r$. 
\end{enumerate}
Then, $G\in \G(qK_k,L)$, where the graph $L$ is the union of 2 copies of $K_k$ sharing $k-r$ vertices. 
\end{observation}

Our proof of \Cref{thm:1} is based on the following procedure (i.e., Procedure~\ref{procedureshrink}) that takes as input a connected graph~$G$ with a $(k,1)$-cover, where each iteration deletes vertices and edges that are incident to a new copy of $K_k$. The lower bound in \Cref{thm:1} is established by estimating the number of edges removed in each iteration of the procedure. This procedure will also be analyzed to characterize all extremal graphs in \Cref{thm:1}. 

\begin{proof}[\bf{Proof of \Cref{thm:1}}] 
    Suppose $G=(V,E)$ is a graph as in \Cref{thm:1}. Apply the following procedure to~$G$.
\begin{algorithm}[H]
    \caption{ } 
    \label{procedureshrink}
    \begin{algorithmic}[1]
           \State \textbf{Initialization:} Let $C_0$ be a copy of $K_k$ in $G$.
           \State Let $V_0 \coloneqq V \setminus V(C_0)$ and  $E_0 \coloneqq E \setminus E(C_0)$. Set $j\coloneqq 1$.
           \While{$V_{j-1}\neq \emptyset$}
           \begin{algsubstates}
             \State \begin{varwidth}[t]{\linewidth}
      Find an edge $e_j=u_jv_j \in E_{j-1}$ with $u_j \in V_{j-1}$ and $v_j \in V \setminus V_{j-1}$.\par
         Let $C_j$ be a copy of $K_k$ in $G$ with $e_j \in E(C_j)$.\par
        Set $V_j \coloneqq V_{j-1} \setminus V(C_j)$ and $E_j \coloneqq E_{j-1} \setminus E(C_j)$.\par
      \end{varwidth}
               \State $j\coloneqq j+1$
           \end{algsubstates}
           \EndWhile
    \end{algorithmic}
\end{algorithm}
Before we proceed with the proof, we make some remarks on Procedure~\ref{procedureshrink}. In line~3(a) of Procedure~\ref{procedureshrink}, such an edge $e_j$ always exists because $G$ is connected. Moreover, since $G$ has a $(k,1)$-cover, such a clique $C_j$ always exists. 
Let us assume that the loop of Procedure~\ref{procedureshrink} runs for $I\geq 0$ many iterations, i.e., $I$ is the smallest number for which $V_I = \emptyset$.

Note that $|V|-|V_0|=k$ and $|E|-|E_0|=\binom{k}{2}$. 
For every $j\in [I]$, let \[x_j:= |V(C_j)\cap (V\setminus V_{j-1})| = k - (|V_{j-1}|-|V_j|).\] 
Observe that $|E_{j-1}|-|E_j|\ge \binom{k}{2}-\binom{x_j}{2}$ for every $j\in [I]$. 
Moreover, since at each iteration $j\in [I]$, the edge $e_j=u_jv_j$ is chosen so that $u_j\in V_{j-1}$ and $v_j \in V \setminus V_{j-1}$ and the clique $C_j$ contains the edge $u_jv_j$, we have $x_j\in [k-1]$. 
By setting $x'_j=x_j-1$ for every $j\in [I]$, we get
\begin{equation}\label{eq:condition on x'_j}
    x'_j\in \{0,\dots, k-2\} \;\; \text{for every} \;\; j\in [I].
\end{equation}
Since $V_I=\emptyset$ and $|V|-|V_0|=k$ and $|V_{j-1}| -|V_j|= (k-x_j)=(k-x'_j-1)$ for every $j\in [I]$, we have 
\begin{align*}
    k+q(k-1)+r=n=|V|&=|V|-|V_{I}|= |V|-|V_0| + |V_0|-|V_1|+\dots+|V_{I-1}|-|V_{I}| \nonumber \\
    &=k + \sum_{j=1}^{I} (k-x'_j-1))=k+  I(k-1)-\sum_{j=1}^{I}x'_j.
\end{align*}
\begin{equation}    
    \text{Thus,} \;\; \sum_{j=1}^{I}x'_j=(I-q)(k-1)-r
    =(I-q-1)(k-1)+(k-r-1).\label{optimizationeq3}
\end{equation}
Similarly, we also have

\begin{align}\label{optimizationeq1}  
    |E|\ge |E|-|E_{I}|&= |E|-|E_0|+ |E_0|-|E_1|+\dots + |E_{I-1}|-|E_{I}|\nonumber \\
    &\ge \binom{k}{2} + \sum_{j=1}^{I} \left(\binom{k}{2}- \binom{x_j}{2}\right)\nonumber \\
    &= \binom{k}{2} + \sum_{j=1}^{I} \left(\binom{k}{2}- \binom{x'_j+1}{2}\right)\text{, subject to \eqref{eq:condition on x'_j} and \eqref{optimizationeq3}.}
\end{align}

We will minimize the expression in \eqref{optimizationeq1}. Since $\sum_{j=1}^{I}x'_j$ is nonnegative and $r\geq 1$, by~\eqref{optimizationeq3}, we must have $I\geq q+1$. At this point, fix $I\geq q+1$. Now, the minimum value of the expression in \eqref{optimizationeq1} is achieved when $\sum_{j=1}^{I}\binom{x'_j+1}{2}$ is maximized. Since $1\leq r \leq k-1$, we have $0\leq k-r-1 \leq k-2$. Using \eqref{optimizationeq3} and in \Cref{corollaryconvex}, setting $q'=I-q-1$ and $r'=k-r-1$ and $m=k-1$ (this relaxes the upperbound for each $x'_j$ in \eqref{eq:condition on x'_j} to $k-1$), we deduce that the maximum value of $\sum_{j=1}^{I}\binom{x'_j+1}{2}$ is achieved if and only if $x'_j\in \{0,k-1\}$ except at most one value of $j$. In particular, the maximum is attained when $x'_j=k-1$ for $j\in [I-q-1]$ and $x'_{I-q}=k-r-1$ and $x'_j=0$ for $j\in \{I-q+1,\dots,I\}$. By inserting these values into \eqref{optimizationeq1} we get
\begin{equation}\label{eq:final line}
|E|\geq \binom{k}{2} + (I-q-1) \left(\binom{k}{2}- \binom{k}{2}\right) + \left(\binom{k}{2}- \binom{k-r}{2}\right) + q \left(\binom{k}{2}- \binom{1}{2}\right)= (q+2)\binom{k}{2} -\binom{k-r}{2}.
\end{equation}
Since this holds for every $I\ge q+1$, we have established the proof of the inequality in \Cref{thm:1}.

We now prove the moreover part of \Cref{thm:1}. 
As already discussed in the introduction, it is easy to see that every graph in the family $\G(qK_k,L)$ achieves equality in \Cref{thm:1}. 
Thus, it is enough to show that every graph achieving equality also belongs to $\G(qK_k,L)$. Subsequently, by \Cref{obs:extremal}, it is enough to prove the following that says that every extremal graph satisfies the hypotheses of \Cref{obs:extremal}. 

\begin{lemma}\label{lem:alternative extremal structure satisfied}
Suppose $G$ is a graph as in \Cref{thm:1} with exactly $(q+2)\binom{k}{2} - \binom{k-r}{2}$ edges. Then, $G$ satisfies the properties in \Cref{obs:extremal}.
\end{lemma}

\begin{proof}
We will first show that $I=q+1$. For the sake of contradiction, fix $I>q+1$. Then, by \Cref{corollaryconvex}, subject to \eqref{optimizationeq3} and $x'_j\in \{0,\dots,k-1\}$ for every $j\in [I]$, the sum $\sum_{j=1}^I \binom{x'_j+1}{2}$ is maximized only if $x'_j = k-1$ for exactly $I-q-1>0$ values of $j$. Thus, using this and \Cref{corollaryconvex} with $q'=I-q-1$ and $r'=k-r-1$ and $m=k-1$, we can conclude that 
\[
\max_{\substack{x'_j\in \{0,\dots,k-2\} \; \text{for} \; j\in [I] \\ x'_1,\dots,x'_I \; \text{satisfies} \; \eqref{optimizationeq3}}} \sum_{j=1}^I \binom{x'_j+1}{2} \;\; < \;\; \max_{\substack{x'_j\in \{0,\dots,k-1\} \; \text{for} \; j\in [I] \\ x'_1,\dots,x'_I \; \text{satisfies} \; \eqref{optimizationeq3}}} \sum_{j=1}^I \binom{x'_j+1}{2} = (I-q-1)\binom{k}{2} + \binom{k-r}{2}.
\]
This together with \eqref{optimizationeq1} yields us $|E| > (q+2)\binom{k}{2} - \binom{k-r}{2}$, a contradiction. 

Thus, it must be the case that $I=q+1$. Since $G$ has exactly $(q+2)\binom{k}{2} - \binom{k-r}{2}$ edges, we have equality in \eqref{eq:final line}. Recall that when we bounded $|E|$ in \eqref{eq:final line}, then the equality is achieved if and only if $x'_j \in \{0,k-1\}$ for all but at most one value of $j\in [q+1]$. Thus, we assume that there is $j'\in [q+1]$ such that $x'_{j'} = k-r-1$ and $x'_j=0$ for every $j\in [q+1]\setminus \{j'\}$. 
Consequently, 
\begin{equation}\label{eq:values of x_j}
    x_{j'} = k-r \;\; \text{and} \;\; x_j=1 \; \text{for every} \; j\in [q+1]\setminus \{j'\}. 
\end{equation}
Since all the inequalities in \eqref{optimizationeq1} must be equality, we have $|E_I|=0$ and ${|E_{j-1}|-|E_j|= \binom{k}{2}-\binom{x_j}{2}}$ for every $j\in [q+1]$. We will next show that $G$ satisfies the hypotheses of \Cref{obs:extremal}. Indeed, by considering the copies $C_0,C_1,\dots,C_{q+1}$ of $K_k$, since ${|E_I|=0}$, \Cref{obs:extremal}(1) holds. By definition of $x_j$ and using \eqref{eq:values of x_j}, we have  
\begin{equation}\label{eq:intersections of C_i}
    |V(C_{j'}) \cap (\cup_{i=0}^{j'-1} V(C_i))| = k-r \;\; \text{and} \;\; |V(C_{j}) \cap (\cup_{i=0}^{j-1} V(C_i))| = 1 \; \text{for every} \; j\in [q+1]\setminus \{j'\}.
\end{equation}
Thus, \Cref{obs:extremal}(2) holds. 
By \eqref{eq:intersections of C_i}, to prove \Cref{obs:extremal}(3), it remains to show 
\begin{equation}\label{eq:exceptional intersection}
    |V(C_i)\cap V(C_{j'})| = k-r \;\; \text{for some} \;\; i\in \{0,\dots,j'-1\}.
\end{equation}
Indeed, this is trivial when $r=k-1$. Thus, we assume $r<k-1$. In this case, suppose for the sake of contradiction, \eqref{eq:exceptional intersection} does not hold. Thus, the set $S = V(C_j') \cap (\cup_{i=0}^{j'-1} V(C_i))$ satisfies $S\not \subseteq V(C_i)$ for every $i\in \{0,\dots,j'-1\}$.
By \eqref{eq:values of x_j} and the discussion after that imply that $|E(C_{j'})\setminus (\cup_{i=0}^{j'} E(C_i))| = |E_{j'-1}|-|E_{j'}|= \binom{k}{2}-\binom{k-r}{2}$. On the other hand, by~\eqref{eq:intersections of C_i}, we have $|E(C_{j'})\setminus \binom{S}{2}|= \binom{k}{2}-\binom{k-r}{2}$, where $\binom{S}{2}$ denotes the set of all pairs of vertices in~$S$. The last two lines together imply that we must have $uv\in \cup_{i=0}^{j'-1} E(C_i)$ for every distinct $u,v\in S$. This is contradicted in the following claim, thus establishing \Cref{obs:extremal}(3). The only thing remaining is to prove the following claim.

\noindent{\bf{Claim.}}
If \eqref{eq:exceptional intersection} does not hold, then there exist $u,v\in S$ such that $uv\notin \cup_{i=0}^{j'-1} E(C_i)$. 

\noindent Let $j^*<j'$ be the minimum index such that $S\subseteq \cup_{i=0}^{j^*} V(C_i)$. Therefore, there exists $u\in S$ such that $u\in V(C_{j^*})$ and $u\notin \cup_{i=0}^{j^*-1} V(C_i)$. Fix such a vertex $u$. By assumption, $S\not \subseteq V(C_{j^*})$. Thus, there exists $v\in S$ such that $v\notin V(C_{j^*})$. Fix such a vertex $v$. Observe that ${uv \notin \cup_{i=0}^{j^*} E(C_i)}$. Moreover, by \eqref{eq:intersections of C_i}, since $u,v\in \cup_{i=0}^{j^*} V(C_i)$, we have $\{u,v\}\not\subseteq V(C_i)$ for every $i\in \{j^*+1,\dots,j'-1\}$. Thus, the edge $uv\notin \cup_{i=0}^{j'-1} E(C_i)$ establishing our claim. This finishes the proof of \Cref{lem:alternative extremal structure satisfied}. 
\end{proof}
\noindent This concludes the proof of \Cref{thm:1}.
\end{proof}

\section{Proof of \Cref{thm:2}} 
\label{section34bound}
In this section, we show that a connected $n$-vertex graph $G$ with a $(3,2)$-cover and with the minimum possible number of edges also has a $(4,1)$-cover. We start by introducing a few notations that will be handy.

\begin{notation}\label{tri}
For an edge $uv$ in a given graph $G$, we denote by $T_G(uv)$ the set of all vertices ${w\in V(G)}$ such that $w$, $u$, and $v$ form a triangle; i.e.,
\[{T_G(uv)\coloneqq\{w \in V(G) \; | \; uw, vw \in E(G)\}}.\] 
We briefly recall the definition of edge contraction.
The \defin{contraction operation} is performed on a specific edge $e=uv$ of a given graph $G$. In this process, the edge 
$e$ is removed, and its endpoints, $u$ and $v$, are identified. 
The graph $G.e$ is defined as the one obtained from $G$ by contracting the edge $e$ and identifying all multi-edges (i.e., changing every multi-edge to a simple edge), see \cref{thm2fig1} for a demonstration.
We note that the contraction of $e=uv$ in $G$ is associated with a map $f\colon V(G) \rightarrow V(G.e)$ defined as
$$        \label{correspond}
f(x)= \begin{cases}x & : x \notin\{u, v\} \\ u_v & : x \in\{u, v\}.\end{cases}
$$  
For $S\subseteq V(G)$, we will often use $f(S)$ to denote the set $\{f(x) : x\in S\}$.
\end{notation}

We will use induction to prove \Cref{thm:2}. The following is our key lemma that will help us execute the induction step.
\begin{lemma}\label{lemtwotriangle}
Let $G=(V, E)$ be a connected graph with a \mbox{$(3,2)$-cover} and $|V|>4$. Suppose $G$ has an edge~$e$ that is not in a copy of $K_4$. Then, the graph $G.e=(V',E')$ is also a connected graph with a \mbox{$(3,2)$-cover} and satisfies $|V'|=|V|-1$ and $|E'| \leq |E| -3$. 
\end{lemma}
\begin{proof}
We first show the bounds on $|V'|$ and $|E'|$. Trivially, $|V'|=|V|-1$. 
Since $G$ has a $(3,2)$-cover, the graph obtained by contracting $e$ has at least two pairs of double-edges. By identifying these multi-edges we remove at least two edges from $G$ in addition to $e$ itself, see \cref{thm2fig1}. Therefore, $|E^{\prime}| \leq |E|-3$. Since $G$ is connected, it is easy to see that $G.e$ is also connected.
\begin{figure}[H]
    \centering
        \subfloat[]{{\includegraphics[scale=0.5]{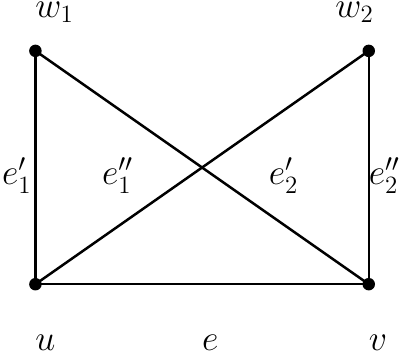}}} \qquad
        \subfloat[]{{\includegraphics[scale=0.54]{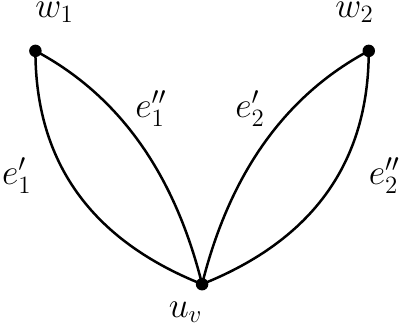}}} \qquad
         \subfloat[]{{\includegraphics[scale=0.54]{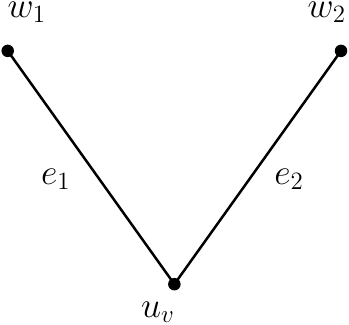}}}
    \caption{The contraction of edge $e=uv$, as described in \cref{lemtwotriangle}.}
    \label{thm2fig1}
\end{figure}
The only thing remaining is to prove that $G.e$ has a \mbox{$(3,2)$-cover}. Let $f$ denote the map corresponding to the contraction of $e$ in $G$.
Thus, it is sufficient to show the following.
\[|T_{G.e}(f(e'))| \geq 2 \;\; \text{for every} \;\; e'\in E\setminus \{e\}.\]
We split the proof into two possible cases depending on whether $e^{\prime}$ is incident to $e$ or not.
  
    \noindent \textbf{Case 1:} We assume that the edges $e^{\prime}=xy$ and $e=uv$ are not incident.\\ 
         \textbf{Subcase 1.1:} If $T_G(e^{\prime}) \cap \{u,v\}=\emptyset$, then $f$ maps every triangle containing $e^{\prime}$ onto itself and so $|T_{G.e}(f(e^{\prime}))| \geq 2$.\\
         \noindent\textbf{Subcase 1.2:} If $|T_G(e^{\prime}) \cap \{u,v\}| =1$, then without loss of generality, assume ${T_G(e^{\prime}) \cap \{u,v\}=\{u\}}$. Since $|T_G(e')|\ge 2$, there exists $z\in T_G(e')\setminus \{u,v\}$. Thus, we have $f(\{x,y,u\}) = \{x,y,u_v\}$ and $f(\{x,y,z\}) = \{x,y,z\}$. Therefore, $z,u_v\in T_{G.e}(f(e^{\prime}))$ and so $|T_{G.e}(f(e^{\prime}))| \geq 2$.\\ 
   \noindent  \textbf{Subcase 1.3:} If $|T_G(e^{\prime}) \cap \{u,v\}|=2$, then $\{x,y\} \subseteq T_G(e)$. This case is impossible since $\{u,v,x, y\}$ induces a copy of $K_4$ in $G$ containing $e$, contradicting our assumption that $e$ is not contained in a copy of $K_4$ in $G$. This case is depicted in \cref{thm2fig3}(a).\\
     \noindent \textbf{Case 2:} We now assume the edges $e^{\prime}=xy$ and $e=uv$ are incident. Without loss of generality, let $y=u$ which means that $e^{\prime}=xu$.\\
        \noindent  \textbf{Subcase 2.1:} If $v \notin T_G(e^{\prime})$, then for every $z \in T_G(e^{\prime})$, we have $f(\{x,u,z\})=\{x,u_v,z\}$, see \cref{thm2fig3}(b) and \cref{thm2fig3}(c). 
        Therefore, since $|T_G(e^{\prime})|\ge 2$, we have $|T_{G.e}(f(e^{\prime}))| \geq 2$.\\
        \begin{figure}[H]
    \centering
            \subfloat[\centering Subcase 1.3]{{\includegraphics[scale=0.75]{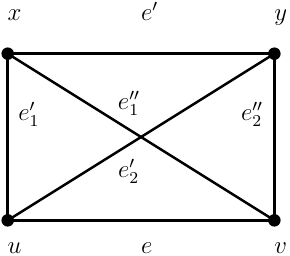}}} \\
        \subfloat[\centering Subcase 2.1 ($G$)]{{\includegraphics[scale=0.8]{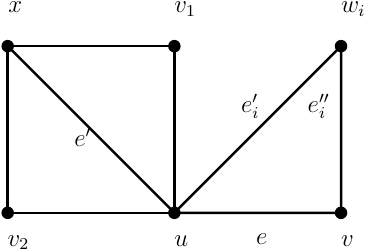}}}\qquad\qquad
        \subfloat[\centering Subcase 2.1 ($G.e$)]{{\includegraphics[scale=0.8]{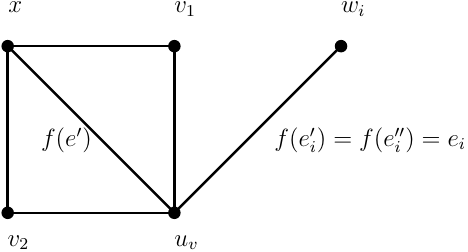}}}\hspace{0.15cm}
          
    \caption{An illustration of Subcase 1.3 and Subcase 2.1 in \cref{lemtwotriangle}.}
    \label{thm2fig3}
    \end{figure}
        \noindent  \textbf{Subcase 2.2:} If $v \in T_G(e^{\prime})$, then $xv\in E$. Therefore, $x \in T_G(e)$. 
        Since $G$ has a $(3,2)$-cover, $ux$ and $vx$ must each be contained in at least one triangle in $G$ other than $\{u,v, x\}$. 
        We claim that $T_G(ux) \cap T_G(vx) = \emptyset$.
        For the sake of contradiction, assume ${z \in T_G(ux) \cap T_G(vx)}$. 
        Then, $\{u,v, x,z\}$ induces a copy of $K_4$ in $G$ containing $e$, a contradiction, see \Cref{thm2fig2}(a). 
        Therefore, $T_G(ux) \cap T_G(vx) = \emptyset $, see \Cref{thm2fig2}(b). 
        In other words, $ux$ and $vx$ belong to two different triangles $\{u, x, z\}$ and $\{v,x, z'\}$ respectively, see \Cref{thm2fig2}(b). 
        Therefore,  $z,z'\in T_{G.e}(f(ux))$, we have $|T_{G.e}(f(e'))|\ge 2$, see \Cref{thm2fig2}(c). 
          \begin{figure}[h]
    \centering
        \subfloat[]{{\includegraphics[scale=0.65]{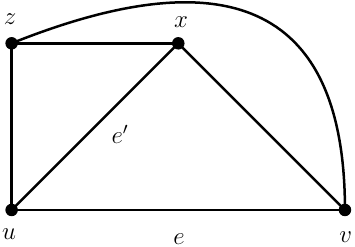}}}\qquad
        \subfloat[]{{\includegraphics[scale=0.65]{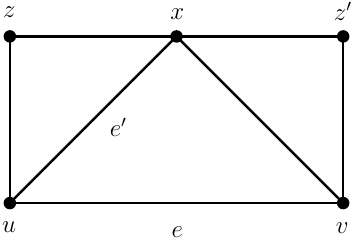}}}\qquad
\subfloat[]{{\includegraphics[scale=0.65]{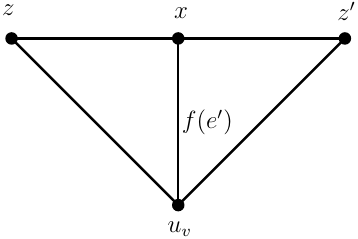}}}
    \caption{An illustration of Subcase 2.2 in \cref{lemtwotriangle}.}
    \label{thm2fig2}
\end{figure}

\noindent This finishes the proof of \Cref{lemtwotriangle}. 
\end{proof}

For the convenience of writing the proof of \Cref{thm:2}, we denote by $F(n)$ the minimum number of edges in a connected $n$-vertex graph with a $(4,1)$-cover. By \Cref{thm:1}, we have $F(n) = 6(q+2) - \binom{4-r}{2}$ where $n-4 = 3q + r$ with $q\ge 0$ and $1\le r\le 3$. 
\begin{lemma}
\label{lem:nminusitwo}
Let $n$, $q$, and $r$ be integers such that $n-4=3q+r$ where $q\ge 0$ and $1\leq r\leq 3$. 
Then, $F(n)-F(n-1)\leq 3$ and $F(n)-F(n-1)= 3$ if and only if $r=1$. 
\end{lemma}
\begin{proof}
       Let $q'\geq 0$ and $1\leq r'\leq 3$ be such that $(n-1)-4=3q'+r'$. We break the proof down into two cases.
    
        \noindent \textbf{Case 1:} $r=1$. In this case, $q^{\prime}=q-1$ and $r^{\prime}=3$. We have
        $$
        \begin{array}{cc}
             &F(n)-F(n-1)= 6(q+2) - \binom{3}{2} -6(q+1)+\binom{1}{2}= 6-3=3.
        \end{array}
        $$
        \noindent   \textbf{Case 2:} $2 \leq r \leq 3$. In this case, $q^{\prime}=q$ and $r^{\prime}=r-1$. We have
       
        \begin{align*}
             F(n)-F(n-1)&= \binom{4-r+1}{2}-\binom{4-r}{2}\in\{1,2\}.
        \end{align*}
\end{proof}

For convenience, we will prove \Cref{thm:2} in the following stronger form. 
\begin{theorem}\label{thm:2 stronger}
If $n\ge 4$ and $G$ is a connected $n$-vertex graph with a \mbox{$(3,2)$-cover} with the minimum possible number of edges, then $G$ has a \mbox{$(4,1)$-cover} and $|E(G)|=F(n)$. 
\end{theorem}

\begin{proof}
We use induction on $n$.\\
\noindent{\textbf{Base step:}} When $n=4$, the only connected $4$-vertex graph with a $(3,2)$-cover is $K_4$ and thus \Cref{thm:2 stronger} holds.\\
\noindent{\textbf{Induction step:}} Let $n\ge 5$. Now suppose \Cref{thm:2 stronger} holds when the number of vertices of~$G$ is $n-1$. 
Let $G=(V,E)$ be a connected $n$-vertex graph with a \mbox{$(3,2)$-cover} with the minimum possible number of edges. If every edge of $G$ is contained in a copy of $K_4$, then we are done. Therefore, we now assume that $G$ has an edge $e=uv$ that is not contained in a copy of $K_4$. By \Cref{lemtwotriangle}, the graph $G.e$ is a connected graph with a \mbox{$(3,2)$-cover} and satisfies $|V(G.e)|=n-1$.

We show the following claim.
\begin{claim}\label{clm:about G.e}
    The graph $G.e$ has the minimum number of edges among all connected $(n-1)$-vertex graphs with a \mbox{$(3,2)$-cover}. Moreover, $n-1 = 4+3q$ for some $q\ge 0$.
    \label{claimkomaki}
\end{claim}

\begin{proof}
    Since $G.e$ has a \mbox{$(3,2)$-cover}, using the induction hypothesis we have \mbox{$ F(n-1) \leq |E(G.e)|$}. 
    By~\Cref{lemtwotriangle}, we have $|E(G.e)| \leq |E| -3$. By \Cref{thm:1} and \Cref{obs}, we have $|E|\le F(n)$. 
    Combining these inequalities with \Cref{lem:nminusitwo}, we have: 
    \begin{equation}\label{eq:comparison after contraction}
        F(n-1) \leq |E(G.e)| \leq |E|-3 \le F(n)-3 \leq F(n-1). 
    \end{equation}
    Thus, all the inequalities above hold with equality. In particular, we have ${|E(G.e)|= F(n-1)}$ and so by the induction hypothesis on the graph $G.e$, we conclude that $G.e$ has the minimum number of edges among all connected $(n-1)$-vertex graphs with a \mbox{$(3,2)$-cover}. 
    We also have $F(n)-F(n-1)=3$. 
    Thus, by \Cref{lem:nminusitwo}, we have $n-1 = 4+3q$ for some $q\ge 0$.\end{proof}
              \begin{figure}[H]
    \centering
     \subfloat[$G.e$]{{\includegraphics[scale=0.57]{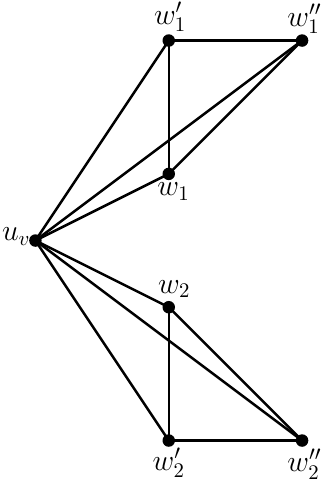}}}\qquad
           \subfloat[]{{\includegraphics[scale=0.57]{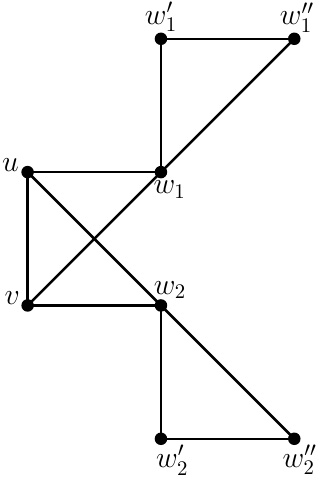}}\qquad} 
          \subfloat[]{{\includegraphics[scale=0.57]{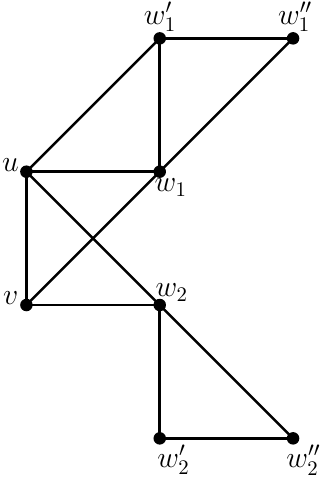}}\qquad}
              \subfloat[]{{\includegraphics[scale=0.57]{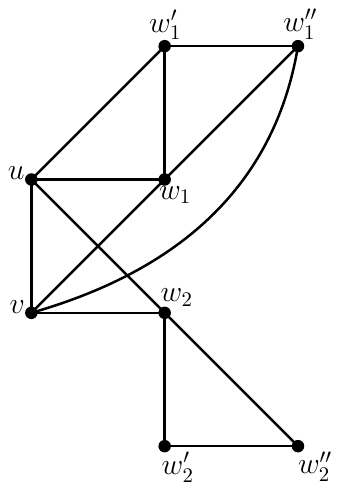}}\qquad}
             \subfloat[]{{\includegraphics[scale=0.57]{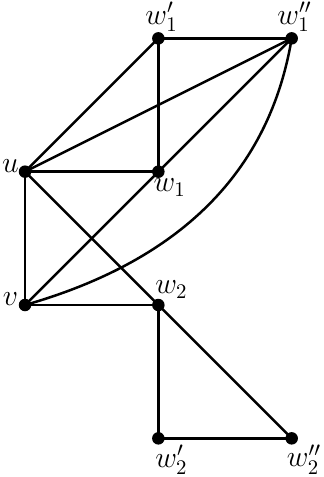}}\qquad}
    \caption{(a): A subgraph of $G.e$ after contracting edge $e$, as described in the proof of \cref{thm:2 stronger}. (b) to (e): A subgraph of $G$ before the contraction of $e=uv$ as described in the proof of \cref{thm:2 stronger}. The subfigures from (b) to (e) are organized incrementally in the sense that each figure updates the previous one with an additional edge. The proof of the existence of each such additional edge in $G$ is described in the proof of~\cref{thm:2 stronger}.}
    \label{thm5fig2}
\end{figure}
By \Cref{clm:about G.e} and the induction hypothesis, the graph $G.e$ has a $(4,1)$-cover and it has the minimum number of edges among all connected $(n-1)$-vertex graphs with a \mbox{$(4,1)$-cover}. Thus, by \Cref{thm:1} and the moreover part of \Cref{clm:about G.e}, we conclude that \mbox{$G.e \in \G((q+1)K_4)$}. 
By \eqref{eq:comparison after contraction}, we have $|E| - |E(G.e)|=3$. Thus, we must have $|T_G(e)|=2$. 
Indeed, if $|T_G(e)|\geq 3$, then we need to identify more than 3 edges from $G$ to get $G.e$, a contradiction.
Let $T_G(e)=\{w_1,w_2\}$. 
Since $e$ does not belong to any copies of $K_4$, we have $w_1 w_2\notin E$. 
Therefore, $G.e$ contains two copies $D_1$ and $D_2$ of $K_4$ with $u_v w_1 \in E(D_1)$ and $u_v w_2 \in E(D_2)$. 
Since $G.e \in \G((q+1)K_4)$, we have $V(D_1) \cap V(D_2) =\{u_v\}$.
Let $V(D_1)=\{u_v, w_1, w^{\prime}_1, w^{\prime \prime}_1\}$ and $V(D_2)=\{u_v, w_2, w^{\prime}_2, w^{\prime \prime}_2\}$. Clearly, $\{w_1, w^{\prime}_1, w^{\prime \prime}_1\}$ and $\{w_2, w^{\prime}_2, w^{\prime \prime}_2\}$ are vertex-disjoint triangles in both $G$ and $G.e$. See \Cref{thm5fig2}(a) and \Cref{thm5fig2}(b).
\noindent Recall from \cref{def_tree} that for any graph in $\G((q+1)K_4)$, the edges belong to exactly one $K_4$ and exactly two triangles. Thus, since $G.e \in \G((q+1)K_4)$, we have the following.
\begin{txteq}\label{uniquebag}
     For every $e'\in E(G.e)$, $\exists$ a unique copy $D$ of $K_4$ in $G.e$ such that ${T_{G.e}(e')\subseteq V(D)}$. 
\end{txteq}

Because $G$ has a $(3,2)$-cover, we know that ${|T_G(u w_1)| \ge 2}$ and $|T_G(vw_1)|\ge 2$. Since $u_v w_1 \in E(D_1)$ (\cref{thm5fig2}(a)), using \eqref{uniquebag}, we have $T_{G.e}(u_v w_1) \subseteq V(D_1)$. 
Thus, $T_G(u w_1)\setminus \{v\} \subseteq\{w^{\prime}_1, w^{\prime \prime}_1\}$ and $T_G(vw_1)\setminus \{u\}\subseteq \{w^{\prime}_1, w^{\prime \prime}_1\}$. Therefore, without loss of generality, we assume $u w^{\prime}_1 \in E$, see~\Cref{thm5fig2}(c). We also have $vw\in E$ for some $w\in \{w^{\prime}_1, w^{\prime \prime}_1\}$. Note that $vw^{\prime}_1$ cannot be an edge of  $G$, because otherwise, we would have $    \{w_1,w_2, w'_1\}\subseteq T_G(e)$, contradicting the fact that $|T_G(e)|=2$. Therefore, we must have $vw^{\prime \prime}_1 \in E$, see \cref{thm5fig2}(d).

We now conclude the proof by showing that $e$ is contained in a copy of $K_4$ and thus obtaining a contradiction. Since $G$ has a \mbox{$(3,2)$-cover}, we have $|T_G(u w^{\prime}_1)|\geq 2$. 
It follows from \eqref{uniquebag} that all triangles in $G.e$ containing $u_v w^{\prime}_1$ lie in $D_1$ and so ${T_G(uw^{\prime}_1) \setminus \{w_1\} = \{w^{\prime \prime}_1\}}$.
Thus, we deduce that $u w^{\prime \prime}_1 \in E$, see~\Cref{thm5fig2}(e). This leads to a contradiction because the vertices $u,v,w_1,w^{\prime \prime}_1$ induce a copy of $K_4$ in $G$ containing the edge $e$, see~\Cref{thm5fig2}(e). 
\end{proof}

\section{Future work}\label{sec:concluding remarks}

The following question remains wide open.
\begin{openproblem}\label{open!}
    For $k\ge 3$ and $\ell \ge 2$, what is the minimum number of edges in a connected $n$-vertex graph with a $(k,\ell)$-cover?
\end{openproblem}

\noindent In this paper, we answered the above question when $k\ge 3,\; \ell =1$ and when $k=3, \;\ell =2$.
Below, we mention a couple of interesting subquestions of \Cref{open!} that will be worth considering. The first one is a generalization of \Cref{thm:2}. 
\begin{openproblem}\label{op1}
    For $k\ge 4$, is it true that every connected $n$-vertex graph with a $(k,2)$-cover with the minimum possible number of edges also has a $(k+1,1)$-cover?
\end{openproblem}

\noindent The next question is another generalization of \Cref{thm:2}. It was already briefly mentioned in the introduction. 
\begin{openproblem}\label{open!!}
    For $\ell \ge 3$, what is the minimum number of edges in a connected $n$-vertex graph with a $(3,\ell)$-cover? 
\end{openproblem}
\noindent In the discussion after \Cref{op4}, we saw that the answer to this question does not coincide with the answer for $(\ell+2,1)$-cover in \Cref{thm:1} for every even $\ell\ge 6$. As mentioned in the introduction, it is known~\cite{BurkhardtFH20} that the asymptotic answer to \Cref{open!!} is $n\left(1+\frac{\ell} {2}\right)+O\left(\ell^2\right)$. However, more work needs to be done to find the exact answer. 

In a different direction, it was considered in \cite{chakraborti2020extremal} to study the minimum number of edges in an $n$-vertex graph where every vertex is in a copy of $H$ for general graphs $H$. 
Similarly, the following question can be studied generalizing the $(k,1)$-cover condition. 
\begin{openproblem}
    For a given graph $H$, what is the minimum number of edges in a connected $n$-vertex graph where every edge is in a copy of $H$?

    \label{problem4}
\end{openproblem}

\noindent{\bf{Acknowledgment}}
We thank Amirhossein Mashghdoust, Pat Morin, and Sophie Spirkl for insightful discussions during the preparation of this paper.

\bibliographystyle{plainurlnat}
\bibliography{main.bib}
\end{document}